\documentclass[a4paper,11pt,reqno]{amsart}

\usepackage[latin1]{inputenc}
\usepackage{amsmath,amssymb,amsthm,amsfonts,mathrsfs,amsopn} 
\usepackage{dsfont}   
\usepackage{mathtools}
\mathtoolsset{showonlyrefs}  

\usepackage{
amssymb,
tikz,
xcolor,
}
\usepackage[hidelinks]{hyperref} 
\usepackage{bbm}
\usepackage{tikz-cd}    
\usepackage{geometry}
\usepackage{bbm}
\usepackage{slashed}

\newtheorem{thm}{Theorem}[section]
\newtheorem{lemma}[thm]{Lemma}
\newtheorem{proposition}[thm]{Proposition}
\newtheorem{corollary}[thm]{Corollary}

\theoremstyle{remark}

\newtheorem{defi}[thm]{Definition}

\numberwithin{equation}{section}

\newcommand{\R}{\mathbb{R}}
\newcommand{\C}{\mathbb{C}}
\newcommand{\dd}{\mathop{}\!\mathrm{d}}

\newcommand{\sph}{\mathbb{S}} 
\newcommand{\I}{\mathrm{i}} 
\newcommand{\eps}{\varepsilon}  
\newcommand{\D}{\slashed{D}} 
\newcommand{\Lap}{L}  
\newcommand{\parenthesis}[1]{\left(#1\right)} 

\DeclareMathOperator{\Spin}{Spin}

\DeclareMathOperator{\SO}{SO}
\DeclareMathOperator{\End}{End}
\DeclareMathOperator{\Id}{Id}

\DeclareMathOperator{\Conf}{Conf} 
\DeclareMathOperator{\YM}{Y} 

\renewcommand{\geq}{\geqslant}
\renewcommand{\leq}{\leqslant}

\newcommand\vol{{\operatorname{vol}}}
\newcommand\N{{\mathbb{N}}}
\newcommand{\be}{\begin{equation}}%
\newcommand{\ee}{\end{equation}}

\title[Dirac-Einstein bubbles]{Some properties of Dirac-Einstein bubbles}

\author[W. Borrelli]{William Borrelli}
\address[W. Borrelli]{Scuola Normale Superiore, Centro De Giorgi, Piazza dei Cavalieri 3, I-56100 , Pisa, Italy.} 
\email{william.borrelli@sns.it}

\author[A. Maalaoui]{Ali Maalaoui}
\address[A. Maalaoui]{Department of mathematics and natural sciences, American University of Ras Al Khaimah, PO Box 10021,
Ras Al Khaimah, UAE.} 
\email{ali.maalaoui@aurak.ac.ae}

\date{\today}

\begin{document}

\begin{abstract}
 We prove smoothness and provide the asymptotic behavior at infinity of solutions of Dirac-Einstein equations on $\R^3$, which appear in the bubbling analysis of conformal Dirac-Einstein equations on spin 3-manifolds. Moreover, we classify ground state solutions, proving that the scalar part is given by Aubin-Talenti functions, while the spinorial part is the conformal image of $-\frac{1}{2}$-Killing spinors on the round sphere $\sph^3$.
 \end{abstract}

\maketitle

{\footnotesize

\medskip

}

\section{Introduction}
Let $(M,g)$ be a three dimensional closed oriented Riemannian manifold. It is well-known \cite[page 87]{LawsonMichelsohn} that such manifold admits a spin structure $\sigma$.

Consider the energy functional $\mathcal{E}$
$$\mathcal{E}(g,\psi)=\int_{M}R_{g}dv_{g}+\int_{M}\langle \D_{g}\psi,\psi \rangle-\langle \psi,\psi\rangle\, \dd\vol_g,$$
where $g$ is a Riemannian metric on $M$, $\psi$ is a spinor in the spin bundle $\Sigma M$ on $M$, $R_g$ is the scalar curvature, $\D_{g}$ is the Dirac operator and $\langle \cdot, \cdot \rangle$ is the compatible Hermitian metric on $\Sigma M$. The functional $\mathcal{E}$ generalizes the classical Hilbert-Einstein functional and it is invariant under the group of diffeomorphisms of $M$ as well. This model was investigated in \cite{Belg, Fin, Kim}, where the authors study it in its full generality and provide some properties of the critical points.\\
In \cite{Maalaoui-Martino-JDE19} the restriction of this functional to a conformal class of the metric was studied. This conformal restriction leads to the functional
\begin{equation}\label{eq:Eaction}
E(u,\psi)=\frac{1}{2}\left(\int_{M}u L_{g}u+\langle \D_{g} \psi,\psi \rangle -|u|^{2}|\psi|^{2}\, \dd\vol_g\right).
\end{equation}
where $L_{g}$ is the conformal Laplacian of the metric $g$. This energy functional can also be seen as the three dimensional version of the super-Liouville equation investigated in \cite{jevnikar2019existence,jost2007super,jost2009energy}. One then can easily see that the critical points of this functional satisfy the following Euler-Lagrange equation
\begin{equation}\label{el}
\left\{\begin{array}{ll}
L_{g} u=|\psi|^{2}u\\
& \text{on } M .\\
\D_{g}\psi=|u|^{2}\psi
\end{array}
\right.
\end{equation}
This system is actually critical since the functional is conformally invariant. This conformal invariance results in a bubbling phenomenon as detailed in \cite{Maalaoui-Martino-JDE19}, where the authors proved the following 

\begin{thm}\label{first}
Let us assume that $M$ has a positive Yamabe constant $Y_g(M)$ and let $(u_{n},\psi_{n})$ be a Palais-Smale sequence for $E$ at level $c$. Then there exist $u_{\infty}\in C^{2,\alpha}(M)$, $\psi_{\infty}\in C^{1,\beta}(\Sigma M)$ such that $(u_{\infty},\psi_{\infty})$ is a solution of $(\ref{el})$, $m$ sequences of points $x_{n}^{1},\cdots, x_{n}^{m} \in M$ such that $\lim_{n\to \infty}x_{n}^{k}= x^{k}\in M$, for $k=1,\dots,m$ and $m$ sequences of real numbers $R_{n}^{1},\cdots, R_{n}^{m}$ converging to zero, such that:
\begin{itemize}
\item[i)]   $\displaystyle u_{n}=u_{\infty}+\sum_{k=1}^{m} v_{n}^{k}+o(1)$  in  $H^{1}(M)$,
\item[ii)]  $\displaystyle \psi_{n}=\psi_{\infty}+\sum_{k=1}^{m}\phi_{n}^{k}+o(1)$ in $H^{\frac{1}{2}}(\Sigma M)$,
\item[iii)] $\displaystyle E(u_{n},\psi_{n})=E(u_{\infty},\psi_{\infty})+\sum_{k=1}^{m}E_{\mathbb{R}^{3}}(U_\infty^{k},\Psi_\infty^{k})+o(1)$,
\end{itemize}
\end{thm}
Thanks to a symmetry trick as in \cite{Maalaoui,maalaoui2012changing}, it is possible to show the existence of sign-changing solutions. This relies mainly on the fact that every non-trivial solution $(U,\Psi)$ on $\R^{3}$ (or on the sphere by stereographic projection) satisfies the following lower bound, proved in \cite{Maalaoui-Martino-JDE19}:
 \be\label{eq:lowerbound}
 E_{\R^{3}}(U,\Psi)\geq \widetilde{\YM}(\sph^3,g_0):=\YM(\sph^3,g_0)\lambda^+(\sph^3,g_0)\,,
 \ee
 where $\YM(\sph^3,g_0)$ is the Yamabe invariant of the round 3-sphere, and 
 \[
 \lambda^+(\sph^3,g_0):=\inf_{g\in[g_0]}\lambda_1(\sph^3,g)\vol(\sph^3,g)^{1/3}
 \]
 is the Hijazi-B\"ar-Lott invariant \cite[\S 8.5]{diracspectrum}. Here $\lambda_1(\sph^3,g)$ denotes the first positive Dirac eigenvalue on $(\sph^3,g)$ and $[g_0]$ is the conformal class of the round metric.

In order for one to have a general existence theorem, one needs to understand these bubbles and classify them as is the case of the Yamabe problem in \cite{caffarelli1989asymptotic}.
\subsection{Main results}
We are then interested in the following coupled system
\be\label{eq:debubbles}
\left\{\begin{aligned}
   -c_3\Delta u=\vert\psi\vert^2 u\,, \\ 
   \D\psi=\vert u\vert^2\psi \,
\end{aligned}\right.\,\qquad \mbox{on $\R^3$.}
\ee
with $c_3=6$, that describes blow-up profiles stated above since $-c_3\Delta=\Lap_{g_{\R^3}}$. The first main result of the paper deals with regularity and asymptotic behavior of distributional solution to \eqref{eq:debubbles}.

\begin{thm}\label{thm:regdecay}
Let $(u,\psi)\in L^{6}(\R^3)\times L^{3}(\Sigma\R^3)$ be a distributional solution to \eqref{eq:debubbles}, then
\be\label{eq:regbubbles}
(u,\psi)\in C^{\infty}(\R^3)\times C^{\infty}(\Sigma\R^{3}) \, .
\ee
Moreover, the following complete asymptotic expansions hold.

There exist $z_\alpha\in\R$, $\alpha\in\N_0^3$, where at least one $z_{\overline{\alpha}}\neq 0$, such that for any $M\in\N$ there is a $C_M<\infty$ such that
\be\label{eq:uexpansion}
\left| u(x) - |x|^{-1} \sum_{|\alpha|_1\leq M} \frac{x^\alpha}{|x|^{|\alpha|_1}} z_\alpha \right| \leq C_M |x|^{-2-M}
\qquad\text{for all}\ |x|\geq 1 \,.
\ee

There are $\zeta_\alpha\in\Sigma\R^3$, $\alpha\in\N_0^3$, where at least one $\zeta_{\overline{\alpha}}\neq 0$, such that for any $M\in\N$ there is a $C_M<\infty$ such that
\be\label{eq:psiexpansion}
\left| \psi(x) - |x|^{-2} \mathcal{U}(x) \sum_{|\alpha|_1\leq M} \frac{x^\alpha}{|x|^{|\alpha|_1}} \zeta_\alpha \right| \leq C_M |x|^{-3-M}
\qquad\text{for all}\ |x|\geq 1 \,,
\ee
where $x^\alpha=x_1^{\alpha_1}x_2^{\alpha_2}x_3^{\alpha_3}$ and $|\alpha|_1 =|\alpha_1|+|\alpha_2|+|\alpha_3|$, and the matrix $\mathcal{U}(x)$ is defined as
\be
\mathcal{U}(x) = \begin{pmatrix}
0 & -\I\frac{x}{|x|}\cdot\boldsymbol\sigma \\ \I \frac{x}{|x|}\cdot\boldsymbol\sigma & 0
\end{pmatrix} \,,\qquad x\in\R^3\setminus\{0\}\,,
\ee
\end{thm}
 By the above theorem one immediately gets the following decay estimates.
 \begin{corollary}
 Under the assumptions of Theorem \ref{thm:regdecay} there holds
 \[
 \vert u(x)\vert \leq \frac{C}{1+\vert x\vert}\,,\qquad \vert \psi(x)\vert\leq \frac{C}{1+\vert x\vert^2}\,,\qquad x\in\R^3\,,
 \]
 for some $C>0$. Moreover, if $u\geq 0$, then there exists $\alpha\geq 0$ such that
 $$\lim_{|x|\to \infty}(1+|x|)u(x)=\alpha,$$
 and if $\alpha=0$ then $u=0$ and $\psi=0$.
 \end{corollary}
 In fact the second assertion of the corollary follows from the strong maximum principle applied to the $u_{\mathcal{K}}$, the solution obtained after inversion (see Section \ref{sec:asymptotics}).
 \smallskip
 
 Define the spaces
 \be\label{eq:D1}
 D^1(\R^3):=\{u\in L^6(\R^3)\,:\, \nabla u\in L^2(\R^3)\}
 \ee 
 and
  \be\label{eq:D1/2}
  D^{1/2}(\Sigma\R^3):=\{\psi\in L^3(\Sigma\R^3)\,:\, \vert\xi\vert^{1/2}\vert\widehat{\psi}(\xi)\vert\in L^2(\R^3)\}\,,
 \ee 
 $\widehat{\psi}$ denoting the Fourier transform of $\psi$. 
 
 We now consider solutions to \eqref{eq:debubbles}, $(u,\psi)\in D^1(\R^3)\times D^{1/2}(\Sigma\R^3)$, corresponding to critical points of the functional at infinity
 \be\label{eq:action}
 E_{\R^{3}}(u,\psi)=\frac{1}{2}\left(\int_{\R^3}c_3\vert\nabla u\vert^2+\langle\D\psi,\psi\rangle-\vert u\vert^2\vert\psi\vert^2 \,\dd x \right)\,.
 \ee
 
  As recalled in the next section, the functional \eqref{eq:action} and the equations \eqref{eq:debubbles} are conformally covariant, so that we can equivalently consider them on the round sphere $(\sph^3,g_0)$. Then the functional reads
  \be\label{eq:sphereaction}
  E_{\sph^3}(u,\psi)=\frac{1}{2}\left(\int_{\sph^3}\overline{u}\Lap_{g_0}u+\langle\D_{g_0}\psi,\psi\rangle-\vert u\vert^2\vert\psi\vert^2 \,\dd\vol_{g_0} \right)\,,
  \ee
  while the equations are given by
  \be\label{eq:spherebubbles}
\left\{\begin{aligned}
   \Lap_{g_0} u=\vert\psi\vert^2 u\,, \\ 
   \D_{g_0}\psi=\vert u\vert^2\psi \,
\end{aligned}\right.\,\qquad \mbox{on $\sph^3$.}
\ee
  \begin{defi}
 We say that a non-trivial solution $(u,\psi)\in D^1(\R^3)\times D^{1/2}(\Sigma\R^3)$ is a \emph{ground state solution} if equality in \eqref{eq:lowerbound} holds, that is
 \be\label{eq:equality}
 E_{\sph^3}(u,\psi)= \widetilde{\YM}(\sph^3,g_0)\,.
 \ee
 \end{defi}

 \begin{thm}\label{thm:groundstates}
 Let $(u,\psi)\in H^1(\sph^3)\times H^{1/2}(\Sigma_{g_0}\sph^3)$ be a ground state solution to \eqref{eq:debubbles}, and assume that $u\geq 0$ .
Then, up to a conformal diffeomorphism, $u\equiv 1$ and $\psi$ is a~$(-\frac{1}{2})$-Killing spinor. 
More precisely, there exists a ~$(-\frac{1}{2})$-Killing spinor $\Psi$ on $(\sph^3,g_0)$ and a conformal diffeomorphism $f\in \Conf(\sph^3,g_0)$ such that
\be
u=\parenthesis{\det(\dd f)}^{\frac{1}{6}}
\ee
and
\begin{equation} 
 \psi=\parenthesis{\det(\dd f)}^{\frac{1}{3}}\beta_{f^*g_0,g_0}(f^*\Psi), 
\end{equation}
where $\beta_{f^*g_0,g_0}$ is the spinor identification for conformal metrics. 
 \end{thm}
 
\begin{corollary}\label{cor:explicitbubbles}
Let $(u,\psi)\in H^1(\R^3)\times H^{1/2}(\Sigma_{g_0}\R^3)$ be a ground state solution to \eqref{eq:debubbles}, with $u\geq 0$. 
Then there exists $\widetilde{\Phi}_0\in\Sigma\R^3$ and $x_0\in\R^3$, ~$\lambda>0$ such that 
 
 \begin{align}\label{eq:scalarbubble}
 u(x)=\parenthesis{\frac{2\lambda}{\lambda^2+|x-x_0|^2}}^{\frac{1}{2}} \,,\qquad x\in\R^3   
\end{align} 
and
 \begin{align}\label{eq:spinorialbubble}
 \psi(x)
 =\parenthesis{\frac{2\lambda}{\lambda^2+|x-x_0|^2}}^{\frac{3}{2}}  \parenthesis{\mathds{1}-\gamma\parenthesis{\frac{x-x_0}{\lambda}}}\widetilde{\Phi}_0 \,,\qquad x\in\R^3\,,
\end{align} 
where $\gamma(\cdot)$ denotes the Clifford multiplication.
\end{corollary}
These ground state solutions are of extreme importance, since one can show that they are non-degenerate, and as a consequence obtain existence of solutions to the perturbed problem on the sphere (see \cite{guidi2020existence} for more details). We remark that the classification of all solutions of \eqref{eq:debubbles} having $u\geq 0$ remains an open problem though. Ground state bubbles for conformally invariant Dirac equations have been recently classified in \cite{borrelli2020ground}. Such equations appear in the blowup analysis of critical Dirac equations on manifolds \cite{Isobecritical} and in the study of the spinorial Yamabe problem and related question, see e.g. \cite{ammannmass, nadineboundedgeometry} and references therein. Recently, two-dimensional critical Dirac equations also attracted a considerable interest as effective models in mathematical physics, see e.g. \cite{shooting,massless}.
 
 \section{Some preliminaries}
 In this section we recall some notions useful in the sequel, for the convenience of the reader. We refer, in particular, to \cite{LawsonMichelsohn} for more details on spin structures and the Dirac operator.
 \subsection{Spin structure and the Dirac operator}
 Let $(M,g)$ be an oriented Riemannian manifold, and let $P_{\SO}(M,g)$ be its frame bundle.
 \begin{defi}
 A \textit{spin structure} on $(M,g)$ is a pair $(P_{\Spin} (M,g),\sigma)$, where $P_{\Spin} (M,g)$ is a $\Spin(n)$-principal bundle and $\sigma:P_{\Spin} (M,g)\rightarrow P_{\SO}(M,g)$ is a 2-fold covering map, which is the non-trivial covering $\lambda:\Spin(n)\rightarrow SO(n)$ on each fiber. 
 \end{defi}
 
In other words, the quotient of each fiber by $\{-1,1\}\simeq\mathbb{Z}_{2}$ is isomorphic to the frame bundle of $M$, so that the following diagram commutes
\begin{center}
 \begin{tikzcd}
 P_{\Spin}(M,g)\arrow[rr,"\sigma"]\arrow[dr]& &P_{\SO}(M,g)\arrow[dl]\\
 &M & 
 \end{tikzcd}
\end{center}
A Riemannian manifold $(M,g)$ endowed with a spin structure is called a \textit{spin manifold}.

In particular, the euclidean space $(\R^{n},g_{\R^n})$ and the round sphere $(\mathbb{S}^{n},g_0)$, with $n\geq 2$, admit a unique spin structure.
 
\begin{defi}
The \emph{complex spinor bundle} $\Sigma M\to M$ is the vector bundle associated to the $\Spin(n)$-principal bundle $P_{\Spin}(M,g)$ via the complex spinor representation of~$\Spin(n)$.  
\end{defi}
The complex spinor bundle~$\Sigma M$ has rank $N=2^{[\frac{n}{2}]}$.  
It is endowed with a canonical spin connection~$\nabla$ (which is a lift of the Levi-Civita connection, denoted by the same symbol) and a Hermitian metric~$g$ which will be abbreviated as~$\left<\cdot, \cdot\right>$ if there is no confusion.  

In particular, the spinor bundle of the euclidean space $\R^n$ is trivial, so that we can identify spinors with vector-valued functions $\psi:\R^n\to\C^N$.
\smallskip

The \emph{Clifford multiplication}~$\gamma\colon TM\to \End_\C(\Sigma M)$ verifies the \emph{Clifford relation}
\begin{equation}
 \gamma(X)\gamma(Y)+\gamma(Y)\gamma(X)=-2g(X,Y)\Id_{\Sigma M}, 
\end{equation}
for any tangent vector fields $X,Y\in\Gamma(TM)$, and is compatible with the bundle metric $g$.

Locally, taking an oriented orthonormal tangent frame $(e_j)$ the Dirac operator is defined as
\[
\D^g\psi:=\gamma(e_j)\nabla^g_{e_j}\psi\,,\qquad \psi\in\Gamma(\Sigma M)\,.
\]

Given~$\alpha\in\C$, a non-zero spinor field~$\psi\in\Gamma(\Sigma M)$ is called~$\alpha$-Killing if 
 \begin{equation}\label{eq:killingequation}
  \nabla^g_X\psi=\alpha\gamma(X)\psi,\qquad \forall X\in\Gamma(TM).  
 \end{equation}

For more information on Killing spinors we refer the reader to~\cite[Appendix A]{diracspectrum}. In the paper \cite{Lu-Pope-Rahmfeld-1999} Killing spinors on the round sphere $(\sph^n,g_0)$ are explicitly computed in spherical coordinates. On $(\sph^n,g_0)$ $\alpha$-Killing spinors only exist for $\alpha=\pm1/2$, and via the stereographic projection the pull-back of the~$\pm\frac{1}{2}$-Killing spinors have the form
\begin{equation}\label{eq:Killing on Rn}
 \Psi(x)
 =\parenthesis{\frac{2}{1+|x|^2}}^{\frac{n}{2}}\parenthesis{\mathds{1}\pm\gamma_{_{\R^n}}(x)}\Phi_0
\end{equation}
where~$\mathds{1}$ denotes the identity endomorphism of the spinor bundle~$\Sigma_{g_{\R^n}}\R^n$ and~$\Phi_0\in \C^N$, see e.g. \cite{ammannmass}.
  \subsection{Conformal covariance}\label{sec:conformalcovariance}
  In this section we recall known formulas that relate the Dirac and conformal Laplace operators for conformally equivalent metrics, see  e.g. \cite{diracspectrum,hitchin}. 
  
To this aim we explicitly label the various geometric objects with the metric~$g$, e.g.~$\Sigma_g M, \nabla^{g}$, $\D^g$, etc. 
  
Let~$f\in C^\infty(M)$ and consider the conformal metric $g_f=e^{2f}g$. 
This induces an isometric isomorphism of spinor bundles
\begin{equation}\label{eq:covariantoperators}
 \beta\equiv \beta_{g,g_f}\colon (\Sigma_{g}M, g)\to (\Sigma_{g_u}M, g_f)\,.
\end{equation}
There holds
\begin{align}\label{eq:conformaldirac}
 \D^{g_f}\beta(e^{-\frac{n-1}{2}f}\psi)
 = e^{-\frac{n+1}{2}f}\beta(\D^g\psi),
\end{align}
\begin{align}\label{eq:conformalpenrose}
 L_{g_f}(e^{-\frac{n-2}{2}f}u)=e^{-\frac{n+2}{2}f}L_g u\,.  
\end{align}

The following quantities are conformally invariant. Setting 
\be\label{eq:rescaling}
\varphi:=\beta(e^{-\frac{n-1}{2}f}\psi)\,,\qquad v:=e^{-\frac{n-2}{2}f}u
\ee
there holds
\begin{equation}
\label{eq:conformalinvariance}
\begin{split}
\int_{M} uL_g u \,\dd{\vol_g} &=\int_{M}vL_{g_f} v\, \dd{\vol_{g_f}}\,,\qquad\int_{M} \langle\D^g \psi,\psi\rangle\,\dd{\vol_g} =  \int_{M} \langle\D^{g_f} \varphi,\varphi\rangle \,\dd{\vol_{g_f}} \\
\int_{M}\vert u\vert^6 \dd{\vol_g}&=\int_{M}\vert v\vert^6 \dd{\vol_{g_f}}\,,\qquad\int_{M} |\psi|^{3}\,\dd{\vol_g} =  \int_{M} |\varphi|^{3} \,\dd{\vol_{g_f}} \\
\int_{M}\vert u\vert^2\vert\psi\vert^2\dd{\vol_g} &=\int_{M}\vert v\vert^2\vert\varphi\vert^2\,\dd{\vol_{g_f}}  \,.
\end{split}
\end{equation}

Consequently the action \eqref{eq:Eaction} is conformally invariant, and hence also equation \eqref{el}. 
\medskip

In particular, by a conformal change of metric \eqref{eq:debubbles} can be reinterpreted as an equation on the round sphere $(\sph^3,g_0)$. Indeed, considering the stereographic projection $\pi:\sph^3\setminus{N}\to\R^3$, where $N\in\sph^3$ is the north pole, we have $(\pi^{-1})^*g_0=\frac{4}{(1+\vert x\vert^2)^2}g_{\R^3}$ with $x\in\R^3$. Then if $(u,\psi)\in D^{1}(\R^3)\times D^{1/2}(\R^3,\C^4)$ is a weak solution to \eqref{el}, consider $(u,\varphi)$ defined in by \eqref{eq:rescaling} with $e^{2f}=\frac{4}{(1+\vert x\vert^2)^2}$. Such pair is in $H^1(\sph^3)\times H^{1/2}(\Sigma_{g_0}\sph^3)$ by \eqref{eq:conformalinvariance}, and it is a weak solutions to \eqref{eq:spherebubbles}. Notice that a priori $(v,\varphi)$ is a weak solution only on $\sph^3\setminus \{N\}$. However, the removability of the singularity at the north pole $N$ can be proved by a cut-off argument similar to the one contained in Section \ref{sec:asymptotics}, see also \cite[Theorem 5.1]{ammannsmallest}. 

Moreover, by \eqref{eq:conformalinvariance} the functional \eqref{eq:Eaction} is also conformally invariant, i.e.
\[
E_{\R^3}(u,\psi)=E_{\sph^3}(v,\varphi)\,.
\]


\section{Regularity and asymptotics}
\subsection{Regularity}

We first need the following (see e.g. \cite{borrellifrank,GrWu} ) Liouville-type result.

\begin{lemma}\label{lem:liouville}
Let $p,q\geq 1$ and assume that $(u,\psi)\in L^p(\R^3)\times L^q(\Sigma\R^3)$ satisfies 
\be
\left\{\begin{aligned}
   -c_3\Delta u=0\,, \\ 
   \D\psi=0 \,
\end{aligned}\right.
\ee
on $\R^3$ in the sense of distributions. Then $u\equiv 0$ and $\D\psi\equiv0$.
\end{lemma}

We now use the above lemma to rewrite the first equation in \eqref{eq:debubbles} in integral form. The \emph{Green's functions} of $-c_3\Delta$ and $\D$ are given, respectively, by
\be\label{greendirac}
G(x-y)=-\frac{1}{4\pi\vert x-y\vert}\,,\qquad \Gamma(x-y)=\frac{\imath}{4\pi}\boldsymbol{\alpha}\cdot\frac{x-y}{\vert x-y\vert^{3}}
\ee
where $\boldsymbol{\alpha}=(\alpha_1,\alpha_2,\alpha_3)$ and the $\alpha_j$ are $4\times 4$ Hermitian matrices satisfying the anticommutation relations
\be\label{anticommutation}
\alpha_{j}\alpha_{k}+\alpha_{k}\alpha_{j}=2\delta_{j,k} \,,\qquad 1\leq j,k\leq 3 \,.
\ee
We choose the $\alpha_j$ of a particular block-antidiagonal form, namely, let $\sigma_1,\sigma_2, \sigma_3$ be $2\times2$ Hermitian matrices satisfying analogous anticommutation relations as in \eqref{anticommutation}, that is,
$$\sigma_{j}\sigma_{k}+\sigma_{k}\sigma_{j}=2\delta_{j,k},\qquad 1\leq j,k\leq 3. $$
Then the matrices
\be\label{sigmaanticommutation}
\alpha_{j}=\begin{pmatrix}0 & \sigma_{j} \\ \sigma_{j} & 0 \end{pmatrix}, \qquad 1\leq j\leq 3,
\ee
satisfy \eqref{anticommutation} and we shall work in the following with this choice. Such matrices exist and form a representation of the \emph{Clifford algebra} of the euclidean space $\R^3$, so that 
\[
\D:=-\imath\boldsymbol{\alpha}\cdot\nabla=-\imath\sum^{3}_{j=1}\alpha_{j}\partial_{x_{j}}
\]

\begin{lemma}\label{integraleq}
If $(u,\psi)\in L^{6}(\R^3)\times L^3(\Sigma\R^3)$ solves \eqref{eq:debubbles} in the sense of distributions, then
\be\label{eq:integralform}
u=G\ast(\vert\psi\vert^{2}u) \,,\qquad \psi=\Gamma\ast(\vert u\vert^2\psi)\,.
\ee
\end{lemma}

\begin{proof}
We note the Green functions have the following weak-Lebesgue integrability, namely, $G\in L^{3,\infty}$ and $\Gamma\in L^{3/2,\infty}$. Since $\psi\in L^{3}$ and $u\in L^6$, we have $\vert\psi\vert^{2}u\in L^{6/5}$ and therefore by the weak Young inequality, the function  
$$
\tilde{u}:=G\ast(\vert\psi\vert^{2}u)
$$
satisfies
$$
\tilde{u}\in L^{6}(\R^3) \,.
$$
Moreover, it is easy to see that
$$
-c_3\Delta\tilde{u} = |\psi|^{2}u
\qquad\text{in}\ \R^3
$$
in the sense of distributions. This implies that
$$
-\Delta (u-\tilde{u})=0
\qquad\text{in}\ \R^3
$$
in the sense of distributions and therefore, by Lemma \ref{lem:liouville}, $u=\tilde{u}$, as claimed. With obvious modifications the argument applies to $\psi$, thus concluding the proof.
\end{proof}

We remark that regularity of weak solutions to \eqref{eq:debubbles} does not follow by standard bootstrap arguments, as we are dealing with critical equations. In the following proposition we follow the strategy of \cite{borrellifrank}, where the authors deal with critical Dirac equations.
\begin{proposition}\label{prop:regularity}
Any distributional solution $(u,\psi)\in L^{6}(\R^3)\times L^3(\Sigma\R^3)$ to \eqref{eq:debubbles} is smooth.
\end{proposition}
\begin{proof}
Since the nonlinearity in \eqref{eq:debubbles} is smooth, it suffices to prove that $(u,\psi)\in L^\infty(\R^3)\times L^\infty(\Sigma\R^3)$, as standard elliptic regularity then applies to give smoothness.

We only deal with $u$ as the argument for the spinorial part $\psi$ follows along the same lines, with straightforward modifications. 
\smallskip

Fix $r>6$. We claim that $u\in L^r(\R^3)$ for $6\leq r<\infty$. To this aim, we prove that there exists $C>0$ such that for all $M>0$ there holds
\be\label{eq:sup}
S_M:=\sup\left\{\left\vert\int_{\R^3}\overline{v}u\dd x\right\vert\,:\, \Vert v\Vert _{r'}\leq1\,,\,\Vert v\Vert_{6/5}\leq M \right\}\leq C\,,
\ee
so that 
\[
\sup\left\{\left\vert\int_{\R^3}\overline{v}u\dd x\right\vert\,:\, \Vert v\Vert _{r'}\leq1\,,\,v\in L^{6/5} \right\}\leq C\,,
\]
and then by density and duality, $u\in L^r$. 

Now fix $M>0$ and let $\eps>0$ be a constant to be chosen later. The function
\[
f_\eps:=\vert\psi\vert^2\mathbbm{1}_{\{\delta\leq\vert\psi\vert\leq\mu\}}
\]
is bounded and supported on a set of finite measure. Moreover, we have
\[
\Vert\vert\psi\vert^2-f_\eps \Vert^{3/2}_{3/2}=\int_{\{\vert\psi\vert<\delta\}\cup\{\vert\psi\vert>\mu\}}\vert\psi\vert^3\dd x<\eps
\]
for suitable $\delta,\mu>0$, as $\psi\in L^{3}$. Define $g_\eps:=\vert\psi\vert^2-f_\eps$ and take $v\in L^{r'}\cap L^{6/5}$, with $\Vert v\Vert_{r'}\leq 1$ and $\Vert v\Vert_{6/5}\leq M$.

By \eqref{eq:integralform}, there holds
\[
\int_{\R^3}\overline{v}u\dd x=\int_{\R^3}\overline{v}(G\ast(f_\eps u))\dd x +\int_{\R^3}\overline{v}(G\ast(g_\eps u))\dd x\,.
\]
Using Fubini's theorem we can rewrite the second integral on the right-hand side, as follows:
\be
\begin{split}
\int_{\R^3}\overline{v}(G\ast(g_\eps u))\dd x &= \int_{\R^3}\dd x\,\overline{v}(x)\int_{\R^3}G(x-y)(g_\eps(y) u(y))\dd y \\
&=\int_{\R^3}\dd y\int_{\R^3}\dd x (G(x-y)\overline{v}(x))\,g_\eps(y)u(y) \\
&= \int_{\R^3}(G\ast\overline{v})g_\eps u\dd y\,.
\end{split}
\ee
Applying again the same argument, and since $u=G\ast(\vert\psi\vert^2)u$, we can further rewrite the last integral to get
\be\label{eq:split}
\int_{\R^3}\overline{v}u\dd x=\int_{\R^3}\overline{v}(G\ast(f_\eps u))\dd x+\int_{\R^3}\overline{h_\eps}u\dd x\,,
\ee
where
\be\label{eq:h}
h_\eps:=\vert\psi\vert^2 G\ast(g_\eps(G\ast v))\,.
\ee
We now estimate the two terms in \eqref{eq:split}. Choosing $s:=\frac{3r}{3+2r}$, the first integral in \eqref{eq:split} can be bounded using the H\"older and Young inequalities
\be\label{eq:firstbound}
\begin{split}
\left\vert \int_{\R^3}\overline{v}(G\ast(f_\eps u))\dd x  \right\vert &\leq\Vert v\Vert_{r'}\Vert G\ast(f_\eps u)\Vert_{r}\leq\Vert v\Vert_{r'}\Vert G\Vert_{3,\infty}\Vert f_\eps u\Vert_{s}\\
&\leq \Vert v\Vert_{r'}\Vert G\Vert_{3,\infty}\Vert f_\eps\Vert_{\frac{6s}{6-s}}\Vert u\Vert_6\leq C_\eps\,,
\end{split}
\ee
where the constant $C_\eps$ depends on $\eps,r,\psi, u$ but not on $M$.

We now turn to the second integral on the righ-hand side of \eqref{eq:split}. By \eqref{eq:h}, H\"older and Young inequalities give
\be\label{eq:first_h_estimate}
\begin{split}
\Vert h_\eps\Vert_{r'}&\leq\Vert \vert\psi\vert^2\Vert_{3/2}\Vert G\ast(g_\eps(G\ast v)) \Vert_{s'}\leq \Vert\vert\psi\vert^2\Vert_{3/2}\Vert G\Vert_{3,\infty}\Vert g_\eps(G\ast v)\Vert_{r'} \\
&\leq\Vert\vert\psi\vert^2\Vert_{3/2}\Vert G\Vert_{3,\infty}\Vert g_\eps\Vert_{3/2}\Vert G\ast v\Vert_{s'}\leq \Vert\vert\psi\vert^2\Vert_{3/2}\Vert G\Vert^2_{3,\infty}\Vert g_\eps\Vert_{3/2}\Vert v\Vert_{r'} \\
&\leq C' \Vert g_\eps\Vert_{3/2}\Vert v\Vert_{r'}\,,
\end{split}
\ee
where the constant $C'>0$ depends on $\psi$. Similarly, we get
\be\label{eq:second_h_estimate}
\begin{split}
\Vert h_\eps\Vert_{6/5}&\leq\Vert\vert\psi\vert^2\Vert_{3/2}\Vert G\ast(g_\eps(G\ast v))\Vert_6\leq\Vert\vert\psi\vert^2\Vert_{3/2}\Vert G\Vert_{3,\infty}\Vert g_\eps(G\ast v)\Vert_{6/5} \\
&\leq\Vert\vert\psi\vert^2\Vert_{3/2}\Vert G\Vert_{3,\infty}\Vert g_\eps\Vert_{3/2}\Vert(G\ast v)\Vert_{6}\leq\Vert\vert\psi\vert^2\Vert_{3/2}\Vert G\Vert^2_{3,\infty}\Vert g_\eps\Vert_{3/2}\Vert v\Vert_{6/5}\\
&\leq C'\Vert g_\eps\Vert_{3/2}\Vert v\Vert_{6/5}\,.
\end{split}
\ee
The estimates \eqref{eq:first_h_estimate} and \eqref{eq:second_h_estimate} imply that 
\[
\left\vert \int_{\R^3}\overline{h}_\eps u\dd x\right\vert\leq C'\Vert g_\eps\Vert_{3/2}S_M\leq C'\eps S_M\,,
\]
by \eqref{eq:sup}. Then taking $\eps=(2C')^{-1}$ and combining the above estimate with \eqref{eq:firstbound} we get
\[
\left\vert\int_{\R^3}\overline{v}u\dd x \right\vert\leq C''+\frac{1}{2}S_M\,,
\]
where $C''$ is the constant $C_\eps$ for $\eps=(2C')^{-1}$. Taking the supremum over all $v$ we have
\[
S_M\leq C''+\frac{1}{2}S_M\implies S_M\leq 2C''\,,
\]
and the claim is proved. A similar argument works for $\psi$, so that we conclude that 
\[
(u,\psi)\in L^r(\R^3)\times L^s(\Sigma\R^3)\,,\qquad \mbox{for all $r\geq6, s\geq 3$}\,,
\]
so that $\vert\psi\vert^2u\in L^t(\R^3)$ for all $t>3/2$. By \eqref{eq:integralform}, we have
\[
u(x)=\int_{\R^3}G(x-y)\vert\psi(y)\vert^2u(y)\dd y\,,
\]
and then writing $G$ as the sum of a function in $L^p$ and one in $L^q$, for some $1<p<3<q$, the H\"older inequality gives $u\in L^\infty$. Arguing along the same lines one proves that $\psi\in L^\infty$, and the proof is concluded.
\end{proof}

\subsection{Existence of asymptotics}\label{sec:asymptotics}
The existence of an asymptotic expansion at infinity for solutions to \eqref{eq:debubbles} follows by the regularity results of the previous section, via the \emph{Kelvin transform} (see \cite[Section 4]{borrellifrank} for the spinorial case). 
Given a function $u$ and a spinor $\psi$ on $\R^3$, it is defined as
\be\label{eq:kelvin}
\left\{\begin{aligned}
    u_{\mathcal{K}}(x):=\vert x\vert^{-1}u(x/\vert x\vert^2)\,, \\ 
   \psi_{\mathcal{K}}(x):=\vert x\vert^{-2}\mathcal{U}(x)\psi(x/\vert x\vert^2) \,
\end{aligned}\right.\qquad x\in\R^3\setminus\{0\}\,,
\ee
where the unitary matrix $\mathcal{U}(x)$ is defined as
\be
\label{eq:defu}
\mathcal{U}(x) = \begin{pmatrix}
0 & -\I\frac{x}{|x|}\cdot\boldsymbol\sigma \\ \I \frac{x}{|x|}\cdot\boldsymbol\sigma & 0
\end{pmatrix} \,,\qquad x\in\R^3\setminus\{0\}\,,
\ee
where $\boldsymbol{\sigma}=(\sigma_1,\sigma_2,\sigma_3)$ (see previous Section).

Let $(u,\psi)\in L^6(\R^3)\times L^3(\Sigma\R^3)$ be a distributional solution to \eqref{eq:debubbles}. Then, as proved in the previous section,  $(u,\psi)\in C^\infty(\R^3)\times C^\infty(\Sigma\R^3)$ and it solves the equations in the classical sense. The transformed couple $(u_{\mathcal{K}},\psi_{\mathcal{K}})$ is smooth as well, and there holds
\be\label{eq:kelvinoperators}
\Delta u_{\mathcal{K}}=\vert x\vert^{-4}(-\Delta u)_{\mathcal{K}}\,,\qquad \D\psi_{\mathcal{K}}=\vert x\vert^{-2}(\D\psi)_{\mathcal{K}}\,.
\ee
The computation for the scalar part is a classical one from potential theory \cite[Section 1.8]{helms-potential}, while the spinorial one can be found in \cite[Section 4]{borrellifrank}. 
Using \eqref{eq:kelvin} and \eqref{eq:kelvinoperators} it is not hardy to see that,
\begin{equation}
\label{eq:inversionlq}
\begin{split}
\int_{\R^3} u_{\mathcal{K}}(-\Delta u_{\mathcal{K}}) \dd x&=\int_{\R^3}u(-\Delta u) \dd x\,,\qquad\int_{\R^3} \langle\D \psi_{\mathcal{K}},\psi_{\mathcal{K}}\rangle\,\dd x =  \int_{\R^3} \langle\D \psi,\psi\rangle \,\dd x \\
\int_{\R^3}\vert u_{\mathcal{K}}\vert^6 \dd x&=\int_{\R^3}\vert u\vert^6 \dd x\,,\qquad\int_{\R^3} |\psi_{\mathcal{K}}|^{3}\,\dd x =  \int_{\R^3} |\psi|^{3} \,\dd x \\
\int_{\R^3}\vert u_{\mathcal{K}}\vert^2\vert\psi_{\mathcal{K}}\vert^2\dd x &=\int_{\R^3}\vert u\vert^2\vert\psi\vert^2\dd x  \,.
\end{split}
\end{equation}
Moreover, by \eqref{eq:kelvinoperators}, one can easily verify that $(u_{\mathcal{K}}, \psi_{\mathcal{K}})$ is smooth and solves \eqref{eq:debubbles} on $\R^3\setminus\{0\}$, and thus is a distributional solution on that set. We now prove that the equation is actually verified in distributional sense on the whole $\R^3$, that is, the origin is a removable singularity.

We only work out the argument for the scalar part $u_{\mathcal{K}}$, as a similar argument applies to $\psi_{\mathcal{K}}$.
We aim at showing that
\be\label{eq:distrsol}
\int_{\R^3}(-\Delta f)u_{\mathcal{K}}\dd x=\int_{\R^3}fu_{\mathcal{K}}\vert\psi_{\mathcal{K}}\vert^2\dd x\,,\qquad\forall f\in C^{\infty}_c(\R^3)\,.
\ee
Let $\eta\in C^\infty(\R^3)$ be (real-valued) a cutoff such that $\eta\equiv1$ outside the unit ball and $\eta\equiv0$ near the origin, and define $\eta_\eps(x):=\eta_\eps(x/\eps)$. Since $\eta_\eps f\in C^\infty_c(\R^3\setminus\{0\})$, there holds
\be\label{eq:outsideorigin}
\int_{\R^3}-\Delta(\eta_\eps f)u_{\mathcal{K}}\dd x=\int_{\R^3}\eta_\eps fu_{\mathcal{K}}\vert\psi_{\mathcal{K}}\vert^2\dd x\,,\qquad\forall f\in C^{\infty}_c(\R^3)\,.
\ee
We have that $\eta_\eps\overline{f}u_{\mathcal{K}}\vert\psi_{\mathcal{K}}\vert^2\in L^1$, $\vert\eta_\eps\vert\leq\Vert\eta\Vert_\infty$ and $\eta_\eps\to1$ almost everywhere, as $\eps\to0$. Then, by dominated convergence, the right-hand side of \eqref{eq:outsideorigin} converges to the second integral in \eqref{eq:distrsol}. Observe that
\be\label{eq:laplacian}
-\Delta(\eta_\eps f)=\eta_\eps(-\Delta f)-2\nabla\eta_\eps\cdot\nabla f-(-\Delta\eta_\eps)f\,.
\ee
Arguing as above, one sees that the contribution to \eqref{eq:outsideorigin} of the first term on the right-hand side of \eqref{eq:laplacian} converges to the first integral in \eqref{eq:distrsol}. 

The other to terms go to zero, as $\eps\to0$. Indeed, there holds
\be
\begin{split}
\left\vert\int_{\R^3}\nabla\eta_\eps\cdot\nabla u_{\mathcal{K}} f\dd x  \right\vert&=\left\vert\int_{\{\vert x\vert<\eps\}}\nabla\eta_\eps\cdot\nabla u_{\mathcal{K}} f\dd x  \right\vert\\ 
&\leq\Vert f\Vert_\infty\Vert\nabla\eta_\eps\Vert_\infty\Vert \nabla u_{\mathcal{K}}\Vert_{2}\vert\{\vert x\vert<\eps\}\vert^{1/2}\lesssim \sqrt{\eps}\,,
\end{split}
\ee
and
\be
\begin{split}
\left\vert\int_{\R^3}(-\Delta\eta_\eps) u_{\mathcal{K}}f\dd x  \right\vert&=\left\vert\int_{\{\vert x\vert<\eps\}}(-\Delta\eta_\eps) u_{\mathcal{K}}f\dd x  \right\vert\\ 
&\leq\Vert f\Vert_\infty\Vert\Delta\eta_\eps\Vert_\infty\Vert u_{\mathcal{K}}\Vert_{6}\vert\{\vert x\vert<\eps\}\vert^{5/6}\lesssim \sqrt{\eps}\,,
\end{split}
\ee
Then by the regularity result of the previous section we conclude that $(u_{\mathcal{K}},\psi_{\mathcal{K}})$ is a smooth solution of \eqref{eq:debubbles} on $\R^3$. Formulas \eqref{eq:uexpansion} and \eqref{eq:psiexpansion} follow by Taylor expanding $u_{\mathcal{K}}$ and $\psi_{\mathcal{K}}$ at the origin and taking the inverse Kelvin transform. Observe that at least one of the coefficients $z_\alpha$ in \eqref{eq:uexpansion} must be non-zero, and the same holds for the $\zeta_\alpha$ in  \eqref{eq:psiexpansion}, as by unique continuation principle a non-trivial solution to \eqref{eq:debubbles} cannot have derivative of arbitrary order vanishing at a point ( see \cite{aronszajn1956unique} for the Laplacian and \cite{Kim-ProcAms95} concerning the Dirac operator).
\section{Classification of ground states}
\begin{proof}[Proof of Theorem \ref{thm:groundstates}]
Let $(u,\psi)\in H^1(\sph^3)\times H^{1/2}(\Sigma_{g_0}\sph^3)$ be a ground state solution to \eqref{eq:spherebubbles}, with $u\geq0$. By the strong maximum principle one infers that actually $u>0$.

Then consider the conformal change of metric
\be\label{eq;newmetric}
g:=\frac{4}{9}u^4g_0\,,\qquad\mbox{on $\sph^3$,}
\ee
and the corresponding isometry of spinor bundles $\beta:\Sigma_{g_0}\sph^3\to\Sigma_{g}\sph^3$

By formulas \eqref{eq:covariantoperators}
\be\label{eq:newspinors}
v:=\sqrt{\frac{3}{2}}\,,\qquad \varphi:=\frac{3}{2}u^{-2}\beta(\psi)\,,
\ee
solves equations \eqref{eq:spherebubbles} in the metric $g$, that is
\be\label{eq:gequation}
\left\{\begin{aligned}
   \Lap_{g} v=\vert \varphi\vert^2 v\,, \\ 
   \D_{g}\varphi=\frac{3}{2}\varphi \,.
\end{aligned}\right.\,\qquad \mbox{on $\sph^3$.}
\ee
By the definition of the Yamabe invariant, we get
\[
\YM(\sph^3,[g_0])\left( \int_{\sph^3}u^6\dd\vol_{g_0}\right)^{1/3}\leq \int_{\sph^3}\overline{u}\Lap_{g_0}u\dd\vol_{g_0}= \int_{\sph^3}u^2\vert\varphi\vert^2\dd\vol_{g_0}\,,
\]
and since $(u,\psi)$ is a ground state, using \eqref{eq:equality} we find
\[
\left( \int_{\sph^3}u^6\dd\vol_{g_0}\right)^{1/3}\leq\lambda^+(\sph^3,g_0)\,.
\]
Observe that the volume of the metric $g$ is 
\[
\vol(\sph^3,g)=\frac{8}{27}\int_{\sph^3}u^6\dd\vol_{g_0}
\]
The Hijazi's inequality \cite{hijazi,hijazi91} gives
\[
\frac{3}{8}\YM(\sph^3,g_0)\leq\frac{9}{4}\vol(\sph^3,g)^{2/3}=\left( \int_{\sph^3}u^6\dd\vol_{g_0}\right)^{2/3}\leq\lambda^+(\sph^3,g_0)^2\,.
\]
But since the last and the first term of these inequalities coincide, we have equality in the Hijazi inequality, and thus the spinor $\varphi$ is Killing, with constant $-\frac{1}{2}$, by the second equation in \eqref{eq:spherebubbles}.

Combining this with the second equation in \eqref{eq:spherebubbles}, we conclude that $\varphi$ is a twistor spinor, so that by \cite[Prop. A.2.1]{diracspectrum}
\[
\frac{9}{4}\varphi=(\D_g)^2\varphi=\frac{3}{8}R_g\varphi\,,
\]
and the scalar curvature of $g$ is then $R_g=R_{g_0}=6$. By a result of Obata \cite{obata1971theconjectures} there exists an isometry
\[
f:(\sph^3,g)\to(\sph^3,g_0)\,,
\]
such that $f^*g_0=g=\frac{4}{9}u^4g_0$. Then we obtain 
\[
\dd\vol_{f^*g_0}=\det(\dd f)=\left( \frac{8}{27}u^6\right)\dd\vol_{g_0}\implies u=\sqrt{\frac{3}{2}}\det(\dd f)^{1/6}\,.
\]
Recall that the isometry $f$ induces an isometry of spinor bundles $F$, so that the following diagram is commutative
\begin{center}
 \begin{tikzcd}
  (\Sigma_{f^*g_0}\sph^3, f^*g_0) \arrow[r, "F"] \arrow[d]& (\Sigma_{g_0}\sph^3,g_0) \arrow[d] \\
  (\sph^3,f^*g_0) \arrow[r, "f"] & (\sph^3,g_0) 
 \end{tikzcd}\,,
\end{center}
where the vertical arrows are the projections defining the spinor bundles.

Then the spinor $\Psi:=F\circ\varphi\circ f^{-1}$ is $-\frac{1}{2}$-Killing with respect to the round metric $g_0$, as this property is clearly preserved by isometries. Thus we can rewrite 
\[
\varphi=F^{-1}\circ\Psi\circ f\equiv f^*\Psi\,,
\]
and
\[
\psi=\frac{2}{3}u^2\beta^{-1}(\varphi)=\det(\dd f)^{1/3}\beta^{-1}(f^*\Psi)\,.
\] 
\end{proof}

\begin{proof}[Proof of Corollary \ref{cor:explicitbubbles}]
Formula \eqref{eq:scalarbubble} is the well-known one for standard bubbles for the classical Yamabe problem \cite{caffarelli1989asymptotic}.

So we only deal with the spinorial part of the solution. As recalled in \eqref{eq:Killing on Rn}, the pullback on $\R^3$ of $-\frac{1}{2}$-Killing spinors on the round sphere is given by
\be\label{eq:killing}
\Psi(x)
 =\parenthesis{\frac{2}{1+|x|^2}}^{\frac{3}{2}}\parenthesis{\mathds{1}-\gamma_{_{\R^3}}(x)}\Phi_0\,,\qquad x\in\R^3\,,
\ee
where $\Phi_0\in\C^4$. The other solutions to the second equation in \eqref{el} are obtained from those of the form \eqref{eq:killing} by applying conformal transformations of the euclidean space $(\R^3,g_{\R^3})$. 
\smallskip

We deal first with the composition of translation and a scaling. For fixed $x_0\in\R^3, \lambda>0$, consider the map $f_{x_0,\lambda}\colon \R^3 \to \R^3$ defined by
\begin{equation}
 f_{x_0,\lambda}(x)\coloneqq \frac{x- x_0}{\lambda}\,,\qquad x\in\R^3\,.
\end{equation}
Then $f^*_{x_0,\lambda}g_{\R^3}=\lambda^{-2}g_{\R^3}$. Recall that the frame bundle is trivial, i.e. $P_{SO}(\R^3,g_{\R^3})=\R^3\times SO(3)$, so that $f_{x_0,\lambda}$ induces a map $\widetilde{f}_{x_0,\lambda}:\R^3\times SO(3)\to\R^3\times SO(3)$, with 
 \[
 \widetilde{f}_{x_0,\lambda}(x,v_1,v_2,v_3)=(f_{x_0,\lambda}(x),v_1,v_2,v_3)\,,
 \]
 acting as the identity on $SO(3)$, which then lifts to a map on $P_{Spin}(\R^3,g_{\R^3})=\R^3\times Spin(3)$ which also acts as the identity on the second component. Since also the spinor bundle is trivial, that is $\Sigma\R^3=\R^3\times\C^4$, we finally get a map $F_{x_0,\lambda}:\R^3\times\C^4\to\R^3\times\C^4$, and the transformation on $\Psi$ is given by
 \begin{align}
 \psi(x)
 =&\beta_{\lambda^{-2}g_{\R^3}, g_{\R^3}} F_{x_0,\lambda}^{-1}\Psi(f_{x_0,\lambda}(x)) \\
 =& \parenthesis{\frac{2\lambda}{\lambda^2+|x-x_0|^2}}^{\frac{3}{2}}
 \beta_{\lambda^{-2}g_{\R^3}, g_{\R^3}} F_{x_0,\lambda}^{-1}
 \parenthesis{\mathds{1}-\gamma_{_{\R^3}}\parenthesis{\frac{x-x_0}{\lambda}}}
 \Phi_0\,,
\end{align}
where $\beta_{\lambda^{-2}g_{\R^3}, g_{\R^3}}$ is the conformal identification of spinors induced by the conformal change of metric.

By the above discussion we can take $\beta_{\lambda^{-2}g_{\R^3}, g_{\R^3}} F_{x_0,\lambda}^{-1}$ to be the identity map, so that \eqref{eq:spinorialbubble} holds. 
\smallskip

To conclude the proof we need to prove that the transformations on spinors of the form \eqref{eq:spinorialbubble}, induced by a rotation on $\R^3$, give another spinor of the same form suitably choosing new parameters. Let $R\in SO(3)$ and consider the induced map $F_R:\Sigma\R^3\to\Sigma\R^3$. Taking a spinor of the form \eqref{eq:spinorialbubble} we get
\be\label{eq:spinorrotation}
\begin{split}
    F^{-1}_R(\psi(Rx))&=\parenthesis{\frac{2\lambda}{\lambda^2+|Rx-x_0|^2}}^{\frac{3}{2}}
 \beta_{\lambda^{-2}g_{\R^3}, g_{\R^3}} F_R^{-1}
 \parenthesis{\mathds{1}-\gamma_{_{\R^3}}\parenthesis{\frac{Rx-x_0}{\lambda}}}
 \Phi_0 \\
 & =\parenthesis{\frac{2\lambda}{\lambda^2+|x-R^{-1}x_0|^2}}^{\frac{3}{2}}
 \parenthesis{\mathds{1}-\gamma_{_{\R^3}}\parenthesis{\frac{x-R^{-1}x_0}{\lambda}}}
 F_R^{-1}(\Phi_0)\,,
    \end{split}
\ee
where we have used the fact that, given $v\in\R^3, \psi\in\Sigma\R^3$, there holds $F_R\parenthesis{\gamma_{_{\R^3}}(v)\psi}= \gamma_{_{\R^3}}(Rv)F_R(\psi)$. Then the claim follows, as \eqref{eq:spinorrotation} gives a spinor of the \eqref{eq:spinorialbubble}, with parameters $R^{-1}x_0\in\R^3$, $\lambda>0$ and $F_R^{-1}(\Phi_0)\in\Sigma\R^3$.
\end{proof}

\end{document}